\documentclass[reqno, a4paper, 10pt]{amsart}

\usepackage{epsfig} 
\usepackage{amsmath} 
\usepackage{amssymb}  
\usepackage{url}
\usepackage[fixpdftex]{xcolor}

\title[]{On Conditions for Asymptotic Stability of Dissipative
Infinite-Dimensional Systems with Intermittent Damping}
\author{Falk M. Hante, Mario Sigalotti and Marius Tucsnak}

\thanks{F. M. Hante is with
Mathematics Center of Heidelberg (MATCH),
 Interdisciplinary Center for Scientific Computing (IWR),
Im Neuenheimer Feld 368,
69120 Heidelberg, Germany.
E-Mail: \url{falk.hante@iwr.uni-heidelberg.de}. \\
\indent M. Sigalotti is with INRIA Saclay--\^Ile-de-France, Team GECO,  and CMAP, UMR 7641, \'Ecole Polytechnique, Route de Saclay,
  91128 Palaiseau Cedex, France. E-Mail:  \url{mario.sigalotti@inria.fr}.\\
\indent  M. Tucsnak is with Institut
\'Elie Cartan (IECN), UMR 7502,
 BP 239, Vand\oe uvre-l\`es-Nancy 54506,  France
 and CORIDA, INRIA Nancy--Grand Est.
E-Mail:  \url{tucsnak@iecn.u-nancy.fr}.
}

\renewcommand{\:}{\mathcal{\colon}}
\newcommand{\laplace}{\Delta}
\newcommand{\NN}{\mathbb{N}}
\newcommand{\ZZ}{\mathbb{Z}}

\newcommand{\RR}{\mathbb{R}}
\newcommand{\CC}{\mathbb{C}}
\renewcommand{\H}{\mathcal{H}}

\newcommand{\la}{\langle}
\newcommand{\ra}{\rangle}
\newcommand{\weakto}{\rightharpoonup}
\newcommand{\weakstarto}{\overset{*}{\rightharpoonup}}

\newcommand{\Lcal}{\mathcal{L}}
\newcommand{\K}{\mathcal{K}}

\theoremstyle{plain}
\newtheorem{theorem}{Theorem}[section]
\newtheorem{proposition}[theorem]{Proposition}
\newtheorem{corollary}[theorem]{Corollary}
\newtheorem{lemma}[theorem]{Lemma}
\newtheorem{hypothesis}[theorem]{Hypothesis}
\newtheorem{definition}[theorem]{Definition}
\theoremstyle{remark}
\newtheorem{remark}{Remark}[section]
\newtheorem{example}{Example}[section]

\begin{document}

\begin{abstract}
We study the asymptotic stability of a dissipative evolution in a Hilbert space subject to intermittent
damping. We observe that, even if the intermittence satisfies a persistent excitation condition,
if the Hilbert space is infinite-dimensional then the system needs not being asymptotically stable (not even in the weak sense).
Exponential stability is recovered under a generalized observability inequality, allowing for time-domains that are not intervals.
Weak asymptotic stability is obtained under a similarly generalized unique continuation principle.
Finally, strong asymptotic stability is proved for intermittences that
do not necessarily satisfy some persistent excitation condition, evaluating
their total contribution to the decay of the trajectories of the damped system.
Our results are discussed using the example of the wave equation,
Schr\"odinger's equation and,
for strong stability, also
the special case of finite-dimensional systems.
\end{abstract}

{
\maketitle
}


\section{Introduction}

Consider a system of the form
$\dot z = Az + Bu$ with $z$ in some (finite- or infinite-dimensional) Hilbert space $H$, $B$ bounded,
and assume that there exists a stabilizing feedback law $u=u^*=K z$. Now consider
 the system
\begin{equation}
\label{1000} \dot z = Az + \alpha(t)Bu,
\end{equation}
where the {\em signal}
$\alpha$ takes values in $[0,1]$ and
$\alpha(t)=0$ for certain times $t$
(i.e., the control may be switched off over possibly non-negligible subsets of time).
Under which conditions
imposed on $\alpha$ is the closed-loop system  \eqref{1000} with
the same control $u^*$ asymptotically stable?
It must be stressed that a complete knowledge of
$\alpha$ (and, in particular, the precise information on the set of times
where it vanishes) would be a too restrictive condition to
impose on $\alpha$. We rather look for conditions valid for a
whole class $\mathcal {G}$ of functions $\alpha$ and, therefore, we
expect the closed-loop systems \eqref{1000} with $u^*$ to be
asymptotically stable for every $\alpha\in {\mathcal {G}}$ (and, possibly, uniformly with respect to all such $\alpha$).
If $\alpha$ takes
the values $0$ and $1$ only, then the system \eqref{1000} actually switches
between the uncontrolled system $\dot z=Az$ and the controlled one $\dot z =
Az+Bu$.

{
If the uncontrolled dynamics are unstable then we should impose on $\alpha$ conditions
guaranteeing a {\em sufficient amount of action} on the system.
Actually, even if they are asymptotically stable, the stability of the overall system is
not guaranteed in general (see \cite{HanteSigalotti2011}).}

The question issued above
may be motivated by some failure
in the transmission from the controller to the plant, leading to instants of time at which the control is switched off, or to some
time-varying phenomenon affecting the efficiency of the control action.
It is also
related to problems
stemming from identification and adaptive control (see, e.g.,
\cite{ABJKKMPR}). In such type of problems, one is lead to consider the stability of  linear systems of the kind $\dot
z = -P(t)z$, $z\in\RR^N$, where the matrix $P(\cdot)$ is
symmetric non-negative definite.
Under which conditions on $P$
is the non-autonomous system stable? An answer for this
particular case can be found in the seminal paper \cite{MORNAR}
which asserts that, if $P\geq0$ is bounded and has
bounded derivative, it is {\em necessary and sufficient}, for the
global {\em exponential stability} of $\dot z=-P(t)z$, that $P$ is also
{\em persistently exciting}, 
 i.e., that there exist
$\mu,T>0$ such that
$$
\int_{t}^{t+T} \xi^T P(s)\xi\,ds \geq \mu,
$$
for all unitary vectors $\xi\in\RR^N$ and all $t\geq 0$.

The notion of persistent excitation, therefore,
appears naturally as  a reasonable additional
assumption on $\alpha$ while studying the stabilization of \eqref{1000}.
The papers  \cite{ChailletChitourLoriaSigalotti2008,ChitourSigalotti2010}, whose results are detailed below, study
the case of finite-dimensional
 systems of the form \eqref{1000} under the assumption
 that
there exist
two positive constants $\mu,T$ such that, for every $t\geq 0$,
\begin{equation}
\label{Tmu}
\int_t^{t+T}\alpha(s)\,ds\geq \mu.
\end{equation}
Given two positive real numbers $\mu\leq T$, we
say that $\alpha$ is a $T$-$\mu$ PE-signal (standing for {\em persistently exciting signal}) if it satisfies \eqref{Tmu}.
Note
that we do not consider here any extra assumption on the regularity
of the PE-signal $\alpha$ (e.g., having a bounded derivative or being piecewise constant).

In \cite{ChailletChitourLoriaSigalotti2008}
it is proved that if $A$ is neutrally stable
(and $(A,B)$ is stabilizable)
then $u^*=-B^T x$ stabilizes \eqref{1000} exponentially, uniformly with respect to
the class of $T$-$\mu$ PE-signals
(see also \cite{ABJKKMPR}).
The results in
\cite{ChailletChitourLoriaSigalotti2008} cover also the first nontrivial case where $A$ is
not stable,
namely the double integrator $\dot z=J_2z+\alpha b_0u$, where $J_2$ denotes the
$2\times2$ Jordan block corresponding to the eigenvalue zero, the control is scalar and $b_0=(0,1)^T$. It is shown that, for every
pair $(T,\mu)$, there exists a
 feedback $u^*=Kx$
such that the corresponding closed-loop system is exponentially stable, uniformly with respect to
the class of $T$-$\mu$ PE-signals.

In
\cite{ChitourSigalotti2010} this last result is extended by proving that
for the single-input case
$$
\dot z = A z +\alpha(t)b u,\qquad u\in\RR,\qquad z\in\RR^N,
$$
 there exists a stabilizer
uniform feedback $u^*=Kx$
for
the class of $T$-$\mu$ PE-signals whenever $(A,b)$ is controllable and the eigenvalues of $A$
have non-positive real part. It is shown, moreover, that there exist controllable pairs $(A,b)$  for which no such stabilizing feedback
exists.

\medskip

The scope of the present paper is to
extend the analysis described above to infinite-dimensional systems.
We focus on the case where $A$ generates
a strongly continuous contraction semigroup and $K=-B^*$, where $B^*$ denotes the adjoint of $B$. This situation  corresponds to the neutrally stable case studied, in the finite-dimensional setting,  in \cite{ChailletChitourLoriaSigalotti2008}.
Recall that the linear feedback control term $Bu=-BB^*z$ is a common choice to stabilize a dissipative linear system
(see \cite{Slemrod1974} and also \cite{Haraux1989}).

The motivating example, illustrating the new
phenomena associated with the new
setting, is the one of a string, fixed at both ends, and damped---when $\alpha(t)=1$---on a proper subdomain. It is not hard to construct (see Example~\ref{exp-wave} for details) an example of periodic traveling wave on which the damping induced by a certain periodic nonzero signal $\alpha$ (hence, satisfying a persistent excitation condition) is ineffective.
Therefore, the counterpart of the finite-dimensional stabilizability result does not hold and additional assumptions have to be made in order to guarantee the stability of the closed-loop system.

The first type of results in this direction (Section~\ref{s:exponential}) concerns exponential stability. We prove that, if there exist $\vartheta,c>0$ such that
\begin{equation}\label{ineqneeded-intro}
 \int_0^\vartheta\alpha(t)\|B^* e^{tA}z_0\|_H^2\,dt \geq c\|z_0\|_H^2,\quad\text{for all}~\mbox{$T$-$\mu$ PE-signal}~\alpha(\cdot),
\end{equation}
then there exist $M \geq 1$ and $\gamma>0$ such that
the solution $z(t)$ of
\begin{equation}
\label{1010} \dot z = Az - \alpha(t)BB^*z,
\end{equation}
satisfies
$$
 \|z(t)\|_H \leq M e^{-\gamma t}\|z_0\|_H
$$
uniformly with respect to $z_0$ and $\alpha$.
 (See Theorem~\ref{thm2}.)
The counterpart of \eqref{ineqneeded-intro} in the unswitched case (i.e., when $\alpha\equiv 1$) is an  observability inequality for the pair $(A,B^*)$. Condition \eqref{ineqneeded-intro}  can actually be seen as a generalized observability inequality.
The proof of Theorem~\ref{thm2} is based on deducing from \eqref{ineqneeded-intro} a uniform decay for the solutions of \eqref{1010} of the squared norm, chosen as Lyapunov function, on time-intervals of length $T$. The conclusion follows from standard considerations on the scalar-valued Lyapunov function (see, for instance, \cite{AeyelsPeuteman1998}).
As an application of the general stability result we consider the example of the
 wave equation  on a  $N$-dimensional domain, damped everywhere.
It should be stressed that generalized observability inequalities of the type discussed here have already been considered in the literature for the heat equation with boundary or locally distributed control (\cite{Fattorini2005,MizelSeidman1997,Phung,Wang2008}).

The second type of results presented in this paper (Section~\ref{s:weak}) deals with weak stability.
We prove that,
if there exists $\vartheta>0$ such that
$$
 \int_0^\vartheta \alpha(s) \|B^* e^{sA}z_0\|^2_H\,ds \ne 0 \qquad \mbox{ for all }z_0\ne 0\mbox{ and all $T$-$\mu$ signal }\alpha,
$$
then the solution $t\mapsto z(t)$ of system \eqref{1010} converges weakly to $0$ in $H$ as $t\to\infty$ for any initial data $z_0 \in H$ and any $T$-$\mu$ PE-signal $\alpha$.
 (See Theorem~\ref{thm:weak}.) The counterpart of such condition in the unswitched case is the unique continuation property, ensuring approximate controllability (see, e.g., \cite{TucsnakWeiss2009}).
 The proof of Theorem~\ref{thm:weak} is based on a compactness argument.
The theorem is applied to the case of a  Schr\"odinger equation with internal control localized on a subdomain. The generalized unique continuation property is then recovered by an analyticity argument (Privalov's theorem) and standard unique continuation (Holmgren's theorem).

Finally, a third type of results (Section~\ref{s:strong}) concerns strong (but not necessarily exponential) stability.
In the spirit of \cite{HarauxMartinezVancostenoble2005}, instead of imposing conditions on $\alpha$ which are satisfied on every time-window of prescribed length, we admit the ``excitations" to be rarefied in time and of variable duration. Stability is guaranteed by asking that the total contribution of the excitations, suitably summed up, is ``large enough".
More precisely, it is proved
that if  there exist $\rho>0$ and a continuous function $c\: (0,\infty) \to (0,\infty)$
such that for all $T>0$,
\eqref{ineqneeded-intro} holds true with $\mu=\rho T$ and $c=c(T)$,
and if there exists a sequence of disjoint intervals  $(a_n,b_n)$
in $[0,\infty)$ with
$\int_{a_n}^{b_n} \alpha(t)\,dt \geq \rho (b_n-a_n)$ and $\sum_{n=1}^\infty c(b_n-a_n)=\infty$, then the solution $z(\cdot)$ of \eqref{1010} satisfies $\|z(t)\|_H \to 0$ as $t \to \infty$.
A function $c(\cdot)$ as above is explicitely found in the case of the
uniformly damped wave equation (it is of order $T^3$ for $T$ small)
and also in the finite-dimensional case, where it is of the same order as the one
computed by Seidman in the unswitched case (\cite{Seidman1988}).
A large literature is devoted to conditions ensuring stability of second order systems with time-varying parameters, mostly but not exclusively in the finite-dimensional setting. Let us mention, for instance,   \cite{Hatvani1996,HatvaniKrisztinTotik1995,PucciSerrin1996,Smith1961} and the already cited paper \cite{HarauxMartinezVancostenoble2005}. Integral conditions in space, instead of in time, guaranteeing stabilizability of systems whose uncontrolled dynamics are given by a contraction semigroup have also been studied. Let us mention, for instance, \cite{Martinez1999,Nakao1996,Tcheugoue1998} for the wave equation and \cite{GuzmanTucsnak2003} for the plate equation. In analogy with the function $c(\cdot)$ introduced above, in the mentioned papers the correct weight has to be used in order to sum up the contributions of the damping coefficients at different points. An interesting question, with possible applications to bang-bang control, would be to combine these two type of results, i.e., to consider controls supported in ``sufficiently large'' measurable subsets of the time-space domain.

\section{Preliminaries}
Let $H$ be a Hilbert space with scalar product $\la\cdot,\cdot\ra_H$. In the following, we study systems modeled as
\begin{equation}\label{sys}
\left\{\begin{aligned}
&\dot{z}(t) = A z(t) + \alpha(t) B u(t)\\
&u(t) = -B^*z(t)\\
&z(0) = z_0
\end{aligned}\right.
\end{equation}
with $A\: H \supset D(A) \to H$ being a (possibly unbounded) linear operator generating a strongly continuous contraction semigroup $\{e^{tA}\}_{t \geq 0}$, $B\:U\to H$ being a bounded linear operator on some Hilbert space $U$, $B^*\: H\to U$ being its adjoint and
$\alpha\: [0,\infty) \to [0,1]$ being some 
signal possibly tuning the feedback control term $Bu(t) = -BB^*z(t)$. 
 Note that under the above assumptions, the operator $A$ is maximal dissipative (see \cite[Proposition~3.1.13]{TucsnakWeiss2009}).

We are interested in conditions on $A$, $B$
and on a class of signals $\mathcal{G}$
ensuring the asymptotic decay of solutions $z\mapsto z(t)$ of \eqref{sys} to the origin in a suitable
sense as time $t$ tends to infinity---independently of the initial data $z_0 \in H$ and of the specific $\alpha(\cdot)$ chosen in $\mathcal{G}$.
The interesting case is when the uncontrolled
evolution does not generate a strict contraction, i.\,e., when $\|e^{tA}\|=1$ for $t \geq 0$, so that the energy of the system may stay constant in the absence of damping.

The classes of signals we mostly deal with are
those defined by persistent excitation conditions.
The latter are defined as follows:
given two positive constants $T$ and $\mu$ satisfying $\mu \leq T$,
we say that a measurable signal $\alpha(\cdot)$ is a \emph{$T$-$\mu$ PE-signal} if it satisfies
\begin{equation}\label{PEcondition}
 \int_t^{t+T} \alpha(s)\,ds \geq \mu,\quad~\text{for all}~t \in [0,\infty).
\end{equation}

We note at this point that we will not make any further smoothness assumption on $\alpha(\cdot)$ for our stability results in this paper.
Thus our analysis takes into account the modeling of abrupt actuator failures up to the extremal case when the system switches between an
uncontrolled evolution when $\alpha(t)=0$ and a fully controlled evolution when $\alpha(t)=1$. Conditions of the type \eqref{PEcondition} however mean
that to some extent the feedback control is active.  

Solutions of \eqref{sys} have to be interpreted in the  mild sense, i.e., for any $t \geq 0$ and $z_0 \in H$, the solution $z(\cdot)$ of \eqref{sys}, evaluated at time $t$, is  given by
$$
 z(t)=e^{tA}z_0-\int_0^t e^{(t-s)A} \alpha(s) BB^*z(s)\,ds.
$$
For any measurable signal $\alpha$ and for any finite time-horizon
$\vartheta \geq 0$, there exists a unique mild solution $z(\cdot)\in C([0,\vartheta];H)$ (see,
e.\,g., \cite{BallMarsdenSlemrod1982}). 
Occasionally, we write $z(t;z_0)$ to indicate
the dependency of the mild solution on the initial data $z_0$.

As recalled in the introduction,  it is shown in \cite{ChailletChitourLoriaSigalotti2008} that for $H=\RR^N$ and
$(A,B)$ controllable (with $A$ a dissipative $N \times N$-matrix)
the solutions of \eqref{sys} satisfy
$$
 \|z(t)\|\leq M e^{-\gamma t}\|z_0\|,\quad t\geq 0
$$
uniformly in $\alpha$ satisfying \eqref{PEcondition},
in the sense that the constants $M$ and $\gamma$ depend only
on $A,B,\mu$ and $T$.
Such result
does not extend in full generality to infinite-dimensional spaces. We see this from the following example with $(A,B)$ being a controllable pair, made of a skew-adjoint (and thus dissipative) operator $A$ and a bounded operator $B$.

\begin{example} (String equation)\label{exp-wave}
Let us consider 
a damped string of length 
normalized to one
with fixed endpoints.
Its dynamics can be described by
\begin{align}
\label{wellengleichung} v_{tt}(t,x)&= v_{xx}(t,x)-\alpha(t)d(x)^2 v_t(t,x),\quad &(t,x)\in (0,\infty) \times (0,1),\\
\label{initial} v(0,x)&= y_0(x),\quad &x\in (0,1),\\
\label{initial2}  v_t(0,x)&= y_1(x),\quad &x\in (0,1),\\
\label{RB} v(t,0) &=v(t,1)= 0,\quad &t\in (0,\infty),
\end{align}
where $d\in L^\infty(0,1)$ and $\alpha\in L^\infty([0,\infty),[0,1])$.

We can express such dynamics as a system of type \eqref{sys} with
$H=U=H^1_0(0,1)\times L^2(0,1)$, $z(t)=(v(t,\cdot),v_t(t,\cdot))$, $A(z_1(t),z_2(t))=(z_2(t),\partial_{xx} z_1(t))$, $B(z_1(t),z_2(t))=(0,d  z_2(t))$.  The operator $A$ is dissipative taking, as norm in $H$,
$$\|(z_1,z_2)\|^2=\|\partial_x z_1\|^2_{L^2(0,1)}+\| z_2\|^2_{L^2(0,1)}.$$

Assume that
\begin{equation}\label{example-nonstability}
d=\chi_\omega
\end{equation}
for some proper subinterval $\omega$ of $(0,1)$.
Then there exist $T\geq\mu>0$, a $T$-$\mu$ PE-signal $\alpha$, and a corresponding nonzero periodic solution. This follows from the results in \cite{MartinezVancostenoble2002} (see also \cite{HarauxMartinezVancostenoble2005})
and can be illustrated by an explicit counterexample expressed in terms of d'Alembert solutions.

Let $\omega=(a,b)$ and assume, without loss of generality, that $b<1$. Set $b'=\frac{1+b}2$.
Take $T=2$ and $\mu=1-b'$.
Then
$$\alpha=\sum_{k=0}^\infty \chi_{[2k-\mu,2k+\mu)}$$
 is a $T$-$\mu$ signal  and
$$v(t,x)=\sum_{k=0}^\infty (\chi_{[b'+2k,1+2k]}(x+t)- \chi_{[-1-2k,-b'-2k]}(x-t))$$
is a periodic, nonzero, mild solution of \eqref{wellengleichung}, \eqref{RB} corresponding to $\alpha$.

Notice, in particular, that even {\em weak} asymptotic stability fails to hold
in this case. \hfill$\diamond$
\end{example}

The scope of the reminder of the paper is
 to understand to which extent the finite-dimensional results obtained
in \cite{ChailletChitourLoriaSigalotti2008}
may be extended to the case
where $H$ is infinite-dimensional.
A crucial remark in this perspective is the following energy decay estimate.
Let
\begin{equation}\label{eq:Vdef}
V(z)=\frac12\|z\|_H^2
\end{equation}
denote the ``energy'' in $H$ and observe that we have
\begin{equation}\label{Vdecay}
 V(z(t+s)) - V(z(t)) \leq - \int_t^{t+s} \alpha(s) \|B^* z(s)\|_U^2\,ds\quad\text{for all}~s \geq 0,
\end{equation}
so that $V(\cdot)$ is non-increasing along trajectories for all signals $\alpha(\cdot)$. This can be shown by a standard approximation argument \cite[Theorem~2.7]{Pazy1983} and
using that $A$ is maximal dissipative.

The estimate provided by the following lemma will be a key tool in the proof of some of the results in this paper.
\begin{lemma}\label{lem:lem0} Let $0 \leq a \leq b < \infty$. Then, for any measurable function $\alpha\: [0,\infty) \to [0,1]$,
the solution $z(\cdot)$ of system \eqref{sys} satisfies
$$
 V(z(b))-V(z(a))  \leq -(2 + 2(b-a)^2\|B\|^4)^{-1} \int_0^{b-a} \alpha(t+a)\|B^* e^{tA}z(a)\|_U^2\,dt.
$$
\end{lemma}
\begin{proof}
Let
\begin{equation*}
\phi_a(t)=e^{(t-a)A}z(a),\quad t\geq a,
\end{equation*}
and $\psi_a(\cdot)$ be the mild solution of
\begin{equation*}\label{syspsithm3}
\left\{\begin{aligned}
 	&\dot{\psi}_a(t) = A\psi_a(t)-\alpha(t) BB^*z(t),\quad t \geq a,\\
	&\psi_a(a) = 0.
\end{aligned}\right.
\end{equation*}
Observe that
\begin{equation}\label{eq:sumphipsilem0}
 \phi_a(t)+\psi_a(t)=e^{(t-a)A}z(a)-\int_{a}^t
 e^{(t-\tau)A}\alpha(\tau)BB^*z(\tau)\,d\tau = z(t),\quad t \geq a,
\end{equation}
and that
\begin{equation}\label{eq:psiestlem0}
 \sup_{\xi \in [a,b]}\|\psi_a(\xi)\|_H^2 \leq (b-a) \|B\|^2
 \int_{a}^{b} \|\alpha(t) B^*z(t)\|_U^2\,dt,
\end{equation}
because, for $\xi \in [a,b]$,
\begin{align*}
\|\psi_a(\xi)\|_H^2 &\leq \left(\int_{a}^{\xi} \|e^{(\xi-t)A}\|_H \|B\|
\|\alpha(t) B^*z(t)\|_U\,dt\right)^2\\ &\leq \|B\|^2 \left(\int_{a}^\xi
\|\alpha(t) B^*z(t)\|_U\,dt \right)^2\\ &\leq (\xi-a)\|B\|^2 \int_{a}^{\xi}
\|\alpha(t) B^*z(t)\|_U^2\,dt,
\end{align*}
where we used that $\|e^{tA}\| \leq 1$.

From inequality \eqref{eq:psiestlem0} we get
\begin{equation}\label{eq:psiest2lem0}
\begin{aligned}
 \int_{a}^{b} &\alpha(t)\|B^*\psi_a(t)\|_U^2\,dt \leq (b-a)
 \|B^*\|^2 \sup_{t \in [a,b]} \|\psi_a(t)\|_H^2\\ & \leq (b-a)^2 \|B\|^2
 \|B^*\|^2 \int_{a}^{b} \|\alpha(t)B^*z(t)\|_U^2\,dt.
\end{aligned}
\end{equation}

Moreover, using \eqref{eq:sumphipsilem0}, we obtain
\begin{equation}\label{eq:phiest1lem0}
\begin{aligned}
 \int_{a}^{b} &\alpha(t) \|B^* \phi_a(t)\|_U^2\,dt = \int_{a}^{b}
 \alpha(t)\|B^*(z(t)-\psi_a(t))\|_U^2\,dt\\ &\leq 2\left( \int_{a}^{b}
 \alpha(t)\|B^*z(t)\|_U^2\,dt + \int_{a}^{b}
 \alpha(t)\|B^*\psi_a(t)\|_U^2\,dt\right).
\end{aligned}
\end{equation}

Plugging \eqref{eq:psiest2lem0} in \eqref{eq:phiest1lem0}, we get
\begin{equation}\label{eq:phiestlem0}
 \int_{a}^{b} \alpha(t) \|B^* \phi_a(t)\|_U^2\,dt \leq 2(1+(b-a)^2
 \|B\|^4)\int_{a}^{b}\|\alpha(t) B^* z(t)\|_U^2\,dt.
\end{equation}

We get from the energy inequality \eqref{Vdecay} combined with
\eqref{eq:phiestlem0} that
\begin{align*}
V(z(b))-V(z(a)) &\leq -\int_{a}^{b}\alpha(t)\|B^*z(t)\|_U^2\,dt\\
&\leq - \frac{1}{2(1+(b-a)^2
 \|B\|^4)}\int_{a}^{b}\alpha(t)\|B^*\phi_a(t)\|_U^2\,dt\\
&=-\frac{1}{2(1+(b-a)^2
 \|B\|^4)}\int_0^{b-a}\alpha(a+t)\|B^* e^{tA}z(a)\|_U^2\,dt,
\end{align*}
concluding the proof.
\end{proof}



Other useful facts which are used repeatedly below  are the following remarks on the class of $T$-$\mu$ PE-signals.
We note that if $\alpha(\cdot)$ is a $T$-$\mu$ PE-signal, then for every $t_0 \geq 0$, the same is true for $\alpha(t_0+\cdot)$. Moreover,
the set of all $T$-$\mu$ PE-signals is weakly-$*$
compact, i.\,e., for any sequence $(\alpha_n(\cdot))_{n\in\NN}$ in this set, there exists a subsequence $(\alpha_{n(\nu)}(\cdot))_{\nu\in\NN}$ such that for some $T$-$\mu$ PE-signal $\alpha_\infty(\cdot)$
\begin{equation}\label{eq:PEcompactness}
 \int_0^\infty \alpha_\infty(s)g(s)\,ds = \lim_{\nu\to\infty} \int_0^\infty \alpha_{n(\nu)}(s) g(s)\,ds\quad\text{for all}~g\in L^1([0,\infty)).
\end{equation}
The existence of a function $\alpha_\infty\in L^\infty([0,\infty),[0,1])$ satisfying \eqref{eq:PEcompactness} follows from the weak-$*$ compactness of $L^\infty([0,\infty),[0,1])$ and one recovers \eqref{PEcondition} for $\alpha_\infty$ by choosing as $g$ in \eqref{eq:PEcompactness} the indicator function of the interval $[t,t+T]$.

\section{Exponential stability under persistent excitation}\label{s:exponential}

We next show that, under the following condition, asymptotic exponential stability holds.

\begin{hypothesis}\label{hyp:thm2}
There exist two constants $c,\vartheta>0$  such that
\begin{equation}\label{ineqneeded}
 \int_0^\vartheta\alpha(t)\|B^* e^{tA}z_0\|_U^2\,dt \geq c\|z_0\|_H^2,\quad\text{for all}~z_0\in H~\text{and all}~T\text{-}\mu~\text{PE-signals}~\alpha(\cdot).
\end{equation}
\end{hypothesis}


\begin{theorem}\label{thm2}
Under Hypothesis~\ref{hyp:thm2}, there exist two constants $M \geq 1$ and $\gamma>0$ such that the mild solution $z(\cdot)$ of system \eqref{sys} satisfies
\begin{equation}\label{eq:thm2est}
 \|z(t)\|_H \leq M e^{-\gamma t}\|z_0\|_H,\quad t\geq 0,
\end{equation}
for any initial data $z_0 \in H$ and any $T$-$\mu$ PE-signal $\alpha(\cdot)$.
\end{theorem}
\begin{proof}
Fix some $T$-$\mu$ PE-signal $\alpha(\cdot)$ and some $s \geq 0$, and define $V$ by \eqref{eq:Vdef}. Lemma~\ref{lem:lem0} with $a=s$ and  $b=s+\vartheta$, where $\vartheta$ is as in Hypothesis~\ref{hyp:thm2}, then yields
\begin{align*}
V(z(s+\vartheta))-V(z(s)) & \leq -\frac{1}{2(1+\vartheta^2 \|B\|^4)}\int_0^{\vartheta}\alpha(t+s)\|B^* e^{tA}z(s)\|_U^2\,dt.
\end{align*}
Again using that $\alpha(\cdot+s)$ is a $T$-$\mu$ PE-signal, Hypothesis~\ref{hyp:thm2} then implies
\begin{equation*}
V(z(s+\vartheta))-V(z(s)) \leq -\frac{c}{(1+\vartheta^2 \|B\|^4)}V(z(s)). 
\end{equation*}
The desired estimate \eqref{eq:thm2est} then follows from standard arguments.
\end{proof}

Example~\ref{ex_wave_2} below illustrates an application of Theorem~\ref{thm2}. We consider again the model of a damped string introduced in Example~\ref{exp-wave},
replacing the localized damping given in \eqref{example-nonstability} (which, as we proved, gives rise to non-stabilizability)
by a damping acting almost everywhere. The argument is presented for
the general case of the damped wave equation (the string corresponding to the case $N=1$).

\begin{example} (Wave equation)\label{ex_wave_2}
Let $N\geq 1$ and consider a $N$-dimensional version of system \eqref{wellengleichung}--\eqref{RB} introduced in Example~\ref{exp-wave}:
\begin{align}
\label{wellengleichung_2}
v_{tt}(t,x)&= \Delta v(t,x)-\alpha(t)d(x)^2 v_t(t,x),\quad &(t,x)\in (0,\infty) \times \Omega,\\
\label{initial_2} v(0,x)&= y_0(x),\quad &x\in \Omega,\\
\label{initial2_2}  v_t(0,x)&= y_1(x),\quad &x\in \Omega,\\
\label{RB_2} v(t,x)& = 0,\quad &(t,x)\in (0,\infty) \times  \partial \Omega,
\end{align}
where $\Omega$ is a bounded domain in $\RR^N$ and $d\in L^\infty(\Omega)$
satisfies
\begin{equation*}
|d(x)|\geq  d_0>0\quad \mbox{for almost all}~x\in \Omega.
\end{equation*}
We claim that in this case Hypothesis~\ref{hyp:thm2} is satisfied with $\vartheta=T$, taking
$H=H^1_0(\Omega)\times L^2(\Omega)$
with norm
$$\|(z_1,z_2)\|^2=\|\nabla z_1\|^2_{L^2(\Omega)}+\| z_2\|^2_{L^2(\Omega)}.$$

Denote by $(\phi_n)_{n\in \NN}$
an orthonormal basis of $L^2(\Omega)$ made of eigenfunctions of the Laplace--Dirichlet
operator on $\Omega$. For each $n\in\NN$, let $\lambda_n> 0$ be
the eigenvalue corresponding to $\phi_n$. Recall that
$\lambda_n$
goes to infinity as $n\to\infty$.

Let $t\mapsto z(t)=(v(t,\cdot),v_t(t,\cdot))$ be a solution of \eqref{wellengleichung_2}--\eqref{RB_2} with initial condition
$(y_0(\cdot),y_1(\cdot))=(\sum_{n\in\NN} a_n\phi_n(\cdot), \sum_{n=1}^\infty \sqrt{\lambda_n} b_n\phi_n(\cdot))$, where $(\sqrt{\lambda_n}a_n)_{n\in \NN}$ and $(\sqrt{\lambda_n}b_n)_{n\in\NN}$ belong to $\ell^2$. By definition, $\|z(0)\|_H^2=\sum_{n\in\NN} \lambda_n(a_n^2+b_n^2)$
and
$$v(t,x)=\sum_{n\in\NN} a_n\phi_n(x)\cos(\sqrt{\lambda_n} t)+\sum_{n=1}^\infty b_n\phi_n(x)\sin(\sqrt{\lambda_n} t).$$

Then
\begin{align*}
 &\int_0^T\alpha(t)\|B^* z(t)\|_U^2\,dt\geq
 d_0^2\int_0^T\int_\Omega\alpha(t)|v_t(x,t)|^2\,dx\,dt\\
 &\quad=d_0^2\int_0^T\int_\Omega \alpha(t)
 \left(\sum_{n\in\NN} \lambda_n (-a_n\sin(\sqrt{\lambda_n} t)+b_n
 \cos(\sqrt{\lambda_n} t))\phi_n(x)\right)^2\,dx\,dt\\
&\quad= d_0^2\sum_{n\in\NN} \lambda_n \int_0^T \alpha(t)( -a_n \sin({ \sqrt{\lambda_n}  t}) + b_n\cos({\sqrt{\lambda_n} t}))^2 dt,
\end{align*}
where we used that, for all $n,m\in \NN$,
\[\int_0^1 \phi_n(x)\phi_m(x)dx = \delta_{nm}.\]
We are left to prove that there exist $c_0>0$ independent of $n\in\NN$, $a_n,b_n\in \RR$ and of
the $T$-$\mu$ signal $\alpha$
such that
\begin{equation}\label{sinis}
\int_0^T \alpha(t) ( -a_n \sin({ \sqrt{\lambda_n} t}) + b_n\cos({\sqrt{\lambda_n} t}))^2 dt
\geq c_0(a_n^2+b_n^2).
\end{equation}
For every $\epsilon\in(0,1)$, let
\begin{equation}\label{eq:AepsDef}
A_n^\epsilon=\{t\in[0,T]\mid | -a_n \sin({ \sqrt{\lambda_n} t}) + b_n\cos({\sqrt{\lambda_n} t})|>\epsilon \sqrt{a_n^2+b_n^2}\}.
\end{equation}
Notice that $-a_n \sin({ \sqrt{\lambda_n} t}) + b_n\cos({\sqrt{\lambda_n} t})=\sqrt{a_n^2+b_n^2}\sin({ \sqrt{\lambda_n} t}+\theta_n)$ for some $\theta_n\in\RR$. Hence,
$$|-a_n \sin({ \sqrt{\lambda_n} t}) + b_n\cos({\sqrt{\lambda_n} t})|\leq \lambda_n \sqrt{a_n^2+b_n^2}|t-t_0|$$
for every $t_0$ belonging to $\{t_0\mid \sin(\lambda_n t_0+\theta_n)=0\}=\frac{\pi}{\lambda_n}\ZZ-\frac{\theta_n}{\lambda_n}$.
In particular, $[0,T]\setminus A_n^\epsilon$ is contained in the
set of points with a distance from $\frac{\pi}{\lambda_n}\ZZ-\frac{\theta_n}{\lambda_n}$ smaller than $\epsilon/\lambda_n$, i.e., in
the union of intervals of length $2\epsilon/\lambda_n$ centered at elements of $\frac{\pi}{\lambda_n}\ZZ-\frac{\theta_n}{\lambda_n}$.
Therefore,
\begin{align}
\nonumber
\text{meas}(A_n^\epsilon)&\geq T-2\frac{\epsilon}{\lambda_n}
\#\left([0,T]\cap \left(\frac{\pi}{\lambda_n}\ZZ-\frac{\theta_n}{\lambda_n}\right)\right)
\geq T-2\frac{\epsilon}{\lambda_n}\left(\frac{T\lambda_n}{\pi}+1\right)
\\&
\geq T\left(1-\frac{2\epsilon}\pi\right)-2\frac\epsilon{\min_{n\in\NN}\lambda_n}.\label{meas--}
\end{align}
Thus, the measure of $A_n^\epsilon$ tends  to $T$ as $\epsilon$ goes to zero, uniformly with respect to the triple $(n,a_n,b_n)$.  In particular, there exists $\bar\epsilon>0$ such that
for every $n\in\NN$, $a_n,b_n\in \RR$ and every $T$-$\mu$ signal $\alpha$,
$$
\int_{A_n^{\bar\epsilon}} \alpha(t)dt\geq \frac\mu 2.
$$

Then
\begin{align*}
\int_0^T \alpha(t)( -a_n\sin({ \sqrt{\lambda_n} t}) + b_n\cos({\sqrt{\lambda_n} t}))^2 dt&\geq {\bar\epsilon}^2\, (a_n^2+b_n^2)\int_{A_n^{\bar\epsilon}} \alpha(t)dt\\
&\geq \frac{\mu{\bar\epsilon}^2}2 (a_n^2+b_n^2),
\end{align*}
proving \eqref{sinis} with $c_0={\mu{\bar\epsilon}^2}/2$. \hfill$\diamond$
\end{example}

\begin{remark}\label{rem:haraux}
The example presented above shows that the sufficient condition for asymptotic
stability of abstract second order evolution equations with \emph{on/off
damping} considered in \cite{HarauxMartinezVancostenoble2005} is not necessary, as detailed here below.
The question of its necessity had been raised in \cite[p. 2522]{FragnelliMugnai2008}.

Extending a result of \cite{Smith1961} for ordinary differential equations,
it was shown in \cite{HarauxMartinezVancostenoble2005} that existence of a
sequence of open disjoint intervals $I_n$ of length $T_n$ such that
\begin{equation}\label{eq:smithharaux}
\sum_{n=1}^\infty m_n T_n \min\left( T_n^2,(1+m_n M_n)^{-1} \right)=\infty
\end{equation}
and existence of constants $M_n \geq m_n>0$ such that
\begin{equation}\label{eq:smithharaux2}
	m_n \leq \alpha(t) \leq M_n,\quad t \in I_n,
\end{equation}
implies asymptotic stability of systems whose prototype is
\eqref{wellengleichung_2}--\eqref{RB_2}.

Taking, for example, $I_n = (s_n,s_n+\frac1n)$ with
\begin{equation*}
s_n=\sum_{k=1}^{n-1} \frac2k
\end{equation*}
and $\alpha(\cdot)$ piecewise constant such that \eqref{eq:smithharaux2} holds
with $m_n=M_n=1$, 
the sum in \eqref{eq:smithharaux} converges, but for $T=2$,
{
\begin{equation*}
\int_t^{t+T} \alpha(s)\,ds \geq \mu,\qquad t \geq 0,
\end{equation*}
for some $\mu>0$, as it easily follows by noticing that $\lim_{t\to +\infty}\int_t^{t+T} \alpha(s)\,ds=1/2$.
\hfill$\diamond$
}
\end{remark}

Another example that one could consider is the
Schr\"odinger equation with internal damping. Because of the infinite speed of propagation of the Schr\"odinger equation, it is a natural question whether, differently form the case of the wave equation,
stability can be achieved by a localized damping.
We are not able to give an answer to this question (detailed below), which we leave as an open problem.

\begin{example} (Schr\"odinger equation)\label{exp_Sch_1D}
 Consider
\begin{align}
 	i y_t(t,x) + y_{xx}(t,x) + i \alpha(t)d(x)^2 y(t,x)& = 0, 	\quad&(t,x)\in(0,\infty)\times(0,1),\label{schroedingerPDE_1D1}\\
 	y(t,0)=y(t,1)&=0	,					\quad& t\in (0,\infty),\label{schroedingerPDE_1D2}\\
 	y(0,x) &= y_0(x),					\quad& t\in(0,1),	\label{schroedingerPDE_1D3}
\end{align}
with $d(\cdot) \in L^\infty(0,1)$ and $\alpha(\cdot)$ being a $T$-$\mu$ PE-signal. Assume that $d=\chi_\omega$ with $\omega=(a,b)$ a nonempty subinterval of $(0,1)$.

In order to write system
\eqref{schroedingerPDE_1D1}--\eqref{schroedingerPDE_1D3} in the form \eqref{sys}, let $H=U=L^2(0,1)$, define $A$ as $Az=iz_{xx}$, acting on $D(A)=\H^2(0,1)\cap\H^1_0(0,1)$, and let $B\:z\mapsto \chi_\omega z$ be the multiplication operator by the function $\chi_\omega=d$.
Then, for $y_0 \in H$, the mild solution $z(\cdot)$ of \eqref{sys} with this choice of $A$ and $B$ corresponds to the weak solution $y(\cdot)$ of \eqref{schroedingerPDE_1D1}--\eqref{schroedingerPDE_1D3} (see \cite{Ball1977}).

Since $A$ is skew-adjoint, in order to apply Theorem~\ref{thm2} we should prove
that Hypothesis~\ref{hyp:thm2} holds true.
More explicitly, we should prove that there exist $\vartheta,c>0$ such that, for each
$z_0\in L^2(0,1)$ and each $T$-$\mu$ PE-signal $\alpha$,
$$ \int_0^\vartheta\int_a^b \alpha(t)|(e^{tA}z_0)(x)|^2\,dx\,dt \geq c\int _0^1
|z_0(x)|^2dx.$$
In order to fix the ideas, let us take $\vartheta=T>\mu$.
The question can be rephrased by asking whether there exists $c>0$ such that
for every $\Xi\subset [0,T]$ of measure equal to $\mu$,
\begin{equation}\label{conj}
 \int_\Xi\int_a^b |\sum_{n\in\NN}\langle
 \phi_n,z_0\rangle_{L^2(0,1)}\phi_n(x)e^{i n^2 \pi^2 t}|^2\,dx\,dt \geq c\int
 _0^1 |z_0(x)|^2dx,
\end{equation}
with $\phi_n(x)=\sqrt{2}\sin(n\pi x)$.
This problem is, up to our knowledge, open.

The question is somehow related with a discussion presented by
Seidman in \cite{Seidman1986}, where it is
conjectured that, among all such sets $\Xi$,
the maximal constant in \eqref{conj} (uniform with respect to $z_0$) is obtained for intervals.

Notice that in the case $\omega=(0,1)$ inequality \eqref{conj} is satisfied because the $L^2$ norm of $z(t)$ is constant with respect to $t$.
Because of the full damping in space, the techniques developed by Fattorini in \cite{Fattorini2005} would also apply, yielding the required generalized observability inequality.
\hfill$\diamond$
\end{example}

\section{Weak stability under persistent excitation}\label{s:weak}

Our main result is that weak asymptotic stability holds when the pair $(A,B)$ has the following
$T$-$\mu$ PE unique continuation property, which weakens Hypothesis~\ref{hyp:thm2}.
\begin{hypothesis}\label{hyp:thm1}\mbox{}
There exists $\vartheta>0$ such that for all $T$-$\mu$ PE-signals $\alpha(\cdot)$
\begin{equation}\label{eq:hypthm1}
 \int_0^\vartheta \alpha(t) \|B^* e^{tA}z_0\|^2_U\,dt = 0 \quad\Rightarrow\quad z_0 = 0.
\end{equation}
\end{hypothesis}

We will prove the following.


\begin{theorem}\label{thm:weak} Under Hypothesis~\ref{hyp:thm1}, the mild solution $t\mapsto z(t)$ of system \eqref{sys} converges weakly to $0$ in $H$ as $t\to\infty$ for any initial data $z_0 \in H$ and any $T$-$\mu$ PE-signal $\alpha(\cdot)$.
\end{theorem}


\begin{proof}
 It suffices to prove that, 
for each $z_0 \in H$ and for each $T$-$\mu$ PE-signal $\alpha(\cdot)$, the \emph{weak $\omega$-limit set}
\begin{equation*}
\begin{aligned}
 \omega(z_0,\alpha(\cdot)) = \{z \in H \mid~&\text{there exists a sequence}~\{s_n\}_{n\in\NN},~s_n\to\infty,~\text{so that}\\&z(s_n;z_0) \weakto z~\text{as}~n\to\infty\}
\end{aligned}
\end{equation*}
is non-empty and, 
taking $\vartheta>0$ as in Hypothesis~\ref{hyp:thm1}, 
\begin{equation}\label{eq:lemvanish}
z_\infty\in\omega(z_0,\alpha(\cdot))
\Rightarrow\exists\; \alpha_\infty\mbox{ $T$-$\mu$ PE-signal s.\,t. }
\int_0^\vartheta \alpha_\infty(t) \|B^* e^{tA} z_\infty\|_U^2\,dt = 0.
\end{equation}
The assertion of the theorem then follows from \eqref{eq:hypthm1}.

Let $z_0 \in H$ and a $T$-$\mu$-persistent excitation signal $\alpha(\cdot)$ be given. Let $z(\cdot;z_0)$ be the unique mild solution of the system \eqref{sys}
and define $V$ as in \eqref{eq:Vdef}.


From the energy inequality \eqref{Vdecay}, one obtains that the weak $\omega$-limit set $\omega(z_0,\alpha(\cdot))$ is non-empty. So let $z_\infty\in\omega(z_0,\alpha(\cdot))$ be an element of the weak $\omega$-limit set and let $\{s_n\}_{n\in\NN}$,~$s_n\to\infty$ be a sequence of times such that $z(s_n;z_0) \weakto z_\infty$ as $n\to\infty$.

We consider the translations
\begin{equation*}
 z_n(s)=z(s +s_n;z_0) \quad \alpha_n(s)=\alpha(s +s_n)
\end{equation*}
and we note that $z_n(\cdot)$ is the mild solution of system \eqref{sys} for the $T$-$\mu$ PE-signal $\alpha_n(\cdot)$ and initial condition $z_n(0)=z(s_n;z_0)$, i.\,e., $z_n(\cdot)$ satisfies
\begin{equation}\label{ShiftsMildSol}
 z_n(s)=e^{s A}z(s_n;z_0)-\int_0^s e^{(s-t)A}\alpha_n(t) B B^* z_n(t)\,dt.
\end{equation}

Therefore, we have the energy estimates
\begin{equation}\label{VdecayShifts}
 V(z_n(s)) - V(z(s_n;z_0)) \leq - \int_0^s \alpha_n(t) \|B^* z_n(t)\|_U^2\,dt\quad\text{for all}~ s \geq 0.
\end{equation}
From \eqref{VdecayShifts} and \eqref{Vdecay} we get
\begin{equation}\label{ShiftsBound}
 \|z_n(s)\|_H \leq \|z(s_n;z_0)\|_H \leq \|z_0\|_H,\quad s\geq 0,
\end{equation}
and thus, for any $\vartheta \geq 0$, we have that $\{z_n(\cdot)\}_{n\in\NN}$ is a bounded subset of
$C([0,\vartheta];H)$.
Choose $\vartheta>0$ as in Hypothesis~\ref{hyp:thm1}.

We claim that
\begin{equation}\label{weakconv}
z_n(s) \weakto z_\infty(s)\mbox{ as }n\to\infty,\quad\text{for all}~s \in [0,\vartheta],
\end{equation}
where $z_\infty(\cdot)$ is the mild solution of the undamped equation
\begin{equation}\label{sysundamped}
\left\{\begin{aligned}
&\dot{z}(s) = A z(s)\\
&z(0) = z_\infty.
\end{aligned}\right.
\end{equation}
Indeed, much as in \cite{BallSlemrod1979}, we can show that $\{z_n\}_{n\in\NN}$ is equicontinuous in $C([0,\vartheta];H_w)$, where $H_w$ is $H$ endowed with the weak topology. To verify this,
let $s_r \searrow s$ in $[0,\vartheta]$ and select some $\psi\in H$. From \eqref{ShiftsMildSol}, we have that
\begin{equation}\label{Est1}
 \begin{aligned}
  |\la z_n(s_r)-z_n(s),\psi \ra| \leq~&\left|\left\la [e^{s_rA}-e^{sA}]z(s_n;z_0),\psi \right\ra \right|\\
					 &~+ \int_0^s\left|\left\la  [e^{(s_r-t)A}-e^{(s-t)A}]\alpha_n(t)BB^*z_n(t),\psi\right\ra\right|\,dt\\
			&~+ \int_s^{s_r} \left|\la e^{(s_r-t)A}\alpha_n(t) BB^*z_n(t),\psi\ra\right|\,dt.
 \end{aligned}
\end{equation}
Moreover, using \eqref{ShiftsBound}, we have that
\begin{equation*}
 \|\alpha_n(t)BB^*z_n(t)\|_H \leq \|\alpha_n(t)\|_\RR \|BB^*\|_{\Lcal(H)}\|z_n(t)\|_H \leq \mathrm{const.}\|z_0\|_H
\end{equation*}
and, as proved in \cite[Theorem~2.3]{BallSlemrod1979},
\begin{equation*}
 a_r=\sup_{\|\phi\|_H \leq 1,~0\leq t\leq s} |\la[e^{(s-t)A}-e^{(s_r-t)A}]\phi,\psi\ra| \to 0\quad\text{as}~r\to\infty.
\end{equation*}
Thus, from \eqref{Est1} we get
\begin{equation*}
   |\la z_n(s_r)-z_n(s),\psi \ra| \leq \mathrm{const.}a_r\|z_0\|_{H}+\mathrm{const.}|s_r-s|,
\end{equation*}
and hence
\begin{equation}\label{UniformEst1}
   |\la z_n(s_r)-z_n(s),\psi \ra| \to 0\quad\text{uniformly as}~r\to\infty.
\end{equation}
Similarly, one shows that \eqref{UniformEst1} holds for $s_r\nearrow s$ in $[0,\vartheta]$. Thus, $\{z_n\}_{n\in\NN}$ is equicontinuous in $C([0,\vartheta];H_w)$. Again using  that $\{z_n(s)\mid n\in\NN,\;s\in[0,\vartheta]\}$ is bounded in $H$ by \eqref{ShiftsBound}, we may view $\{z_n\}_{n\in\NN}$ as an equibounded set of curves  in $H$ endowed with the metrized weak topology. Hence we can apply
the Arzela--Ascoli theorem for metric spaces to conclude that there exists $z_\infty(\cdot)\in C([0,\vartheta];H_w)$
and a subsequence that we re-label by $n\in\NN$ so that $z_n(s) \weakto z_\infty(s)$ uniformly on $[0,\vartheta]$ as $\nu\to\infty$. Moreover, for any $\psi \in H$ we have from \eqref{ShiftsMildSol}
by adding and subtracting $\la e^{(s-t)A} \alpha_n(t) B B^* z_\infty(t),\psi\ra$ under the integral that
\begin{equation}\label{ShiftsMildSolRearranged}
\begin{aligned}
 \la z_n(s),\psi\ra =~&\la e^{sA}z(s_n;z_0),\psi\ra - \int_0^s \alpha_n(t) \la e^{(s-t)A}  B B^* z_\infty(t),\psi\ra\,dt\\
& -\int_0^s \alpha_n(t) \la e^{(s-t)A}  B B^* [z_n(t)-z_\infty(t)],\psi\ra\,dt.
\end{aligned}
\end{equation}
Using that $\alpha_n(t)$ is a bounded sequence for 
$t\in[0,s]$ and that
\begin{equation*}
\la e^{(s-t)A}  B B^* [z_n(t)-z_\infty(t)],\psi\ra \to 0\quad\text{as}~\nu\to\infty
\end{equation*}
for all $t\in[0,s]$, we can conclude from the dominated convergence
theorem that
\begin{equation*}
 \int_0^s \alpha_n(t) \la e^{(s-t)A}  B B^* [z_n(t)-z_\infty(t)],\psi\ra\,dt \to 0,\quad\text{as}~n\to\infty.
\end{equation*}
Hence, by sequential weak$^*$-compactness of $L^\infty([0,\infty);[0,1])$, we can extract another subsequence that we again re-label by $n\in\NN$ and pass to the limit in \eqref{ShiftsMildSolRearranged}, obtaining that, for every $s \geq 0$,
\begin{equation}\label{ShiftsWeakLimit}
 \la z_\infty(s),\psi\ra =\la e^{sA}z_\infty,\psi\ra-\int_0^s \la e^{(s-t)A}\alpha_\infty(t) B B^* z_\infty(t),\psi\ra\,dt,
\end{equation}
where $\alpha_\infty(\cdot)$ again is  a $T$-$\mu$ PE-signal. Since \eqref{ShiftsWeakLimit} holds for all $\psi \in H$, we have
\begin{equation*}
 z_\infty(s) = e^{sA}z_\infty-\int_0^s e^{(s-t)A}\alpha_\infty(t) B B^* z_\infty(t)\,dt.
\end{equation*}

Next we show that
\begin{equation}\label{intvanishes}
 \int_0^s e^{(s-t)A}\alpha_\infty(t) B B^* z_\infty(t)\,dt=0.
\end{equation}

Since $V(z_n(0))$ is bounded and monotone, it has a limit $V^*=\lim_{n\to\infty} V(z_n)$, so that
\begin{equation*}
 \int_0^\vartheta \alpha_n(t)\|B^* z_n(t)\|_U^2\,dt \leq V(z_n(0)) - V^* \to 0\quad\text{as}~n\to\infty.
\end{equation*}
Hence
\begin{equation}\label{inttozero}
 \int_0^\vartheta \alpha_n(t)\|B^* z_n(t)\|_U^2\,dt \to 0\quad\text{as}~n\to\infty.
\end{equation}
Morover, \eqref{weakconv} and $\alpha_n \weakstarto \alpha_\infty$ imply 
\begin{equation}\label{liminfbnd}
 \liminf_{n\to\infty} \int_0^\vartheta \alpha_n(t)\|B^* z_n(t)\|_U^2\,dt \geq \int_0^\vartheta \alpha_\infty(t)\|B^*z_\infty(t)\|_U^2\,dt.
\end{equation}
To see this, observe that \eqref{weakconv} implies
\begin{equation}\label{normest}
 \|B^*z_\infty(t)\|_U \leq \liminf_{n\to\infty} \|B^*z_n(t)\|_U,\quad~\text{for all}~t\in[0,\vartheta].
\end{equation}
Fix any $\epsilon>0$ and define, for all $m\in\NN$,
\begin{equation*}
 S^\epsilon_m=\{t\in[0,\vartheta] \mid \|B^*z_n(t)\|_U^2 \geq \|B^*z_\infty(t)\|_U^2-\epsilon~\text{for all}~n \geq m\}.
\end{equation*}
From \eqref{normest} we have
\begin{equation*}
 [0,\vartheta]=\bigcup_m S^\epsilon_m\quad (S^\epsilon_m \supseteq S^\epsilon_{m-1}),
\end{equation*}
hence there exists $m(\epsilon)$ such that $|S^{\epsilon}_{m(\epsilon)}| > T-\epsilon$.
Then we have, for $n \geq m(\epsilon)$,
\begin{equation}\label{threeI}
\begin{aligned}
 \int_0^T\alpha_n(t)\|B^*z_n(t)\|_U^2\,dt =~&\int_{S^\epsilon_{m(\epsilon)}} \alpha_n(t)\left( \|B^*z_n(t)\|_U^2 - \|B^*z_\infty(t)\|_U^2 +\epsilon\right)\,dt\\
&+\int_{S^\epsilon_{m(\epsilon)}} \alpha_n(t)\left( \|B^*z_\infty(t)\|_U^2 -\epsilon\right)\,dt\\
&+\int_{[0,\vartheta] \setminus S^\epsilon_{m(\epsilon)}} \alpha_n(t)\|B^*z_n(t)\|_U^2\,dt.
\end{aligned}
\end{equation}
The first integral in the right-hand side of \eqref{threeI}
is non-negative because $\|B^*z_n(t)\|_U^2 - \|B^*z_\infty(t)\|_U^2 +\epsilon \geq 0$
for all $t\in S^\epsilon_{m(\epsilon)}$ and $\alpha_n(t)\geq 0$ for all $t\in[0,\vartheta]$. The second integral is bounded from below by
\begin{equation*}
 \int_{S^\epsilon_{m(\epsilon)}} \alpha_n(t) \|B^*z_\infty(t)\|_U^2\,dt - \epsilon\vartheta \geq \int_0^\vartheta \alpha_n(t) \|B^*z_\infty(t)\|_U^2\,dt - \epsilon(\text{const.} + \vartheta)
\end{equation*}
as it follows from \eqref{ShiftsBound}. Finally, the third integral is non-negative, again because $\alpha_n(t)\geq 0$ for all $t\in[0,\vartheta]$. Thus, for all $n \geq m(\epsilon)$,
\begin{equation*}
 \int_0^\vartheta \alpha_n(t)\|B^*z_n(t)\|^2_U\,dt \geq \int_0^\vartheta \alpha_n(t)\|B^*z_\infty(t)\|_U^2\,dt-\epsilon(C+\vartheta).
\end{equation*}
Hence, by the convergence $\alpha_n(\cdot)\weakstarto\alpha_\infty(\cdot)$,
\begin{equation*}
 \liminf_{n\to\infty} \int_0^\vartheta \alpha_n(t)\|B^* z_n(t)\|_U^2\,dt \geq \int_0^\vartheta \alpha_\infty(t)\|B^*z_\infty(t)\|_U^2\,dt -\epsilon(C+T),
\end{equation*}
proving \eqref{liminfbnd} from the fact that $\epsilon$ is arbitrary.

From \eqref{liminfbnd} and \eqref{inttozero}, we have
\begin{equation}\label{intvanishesnorm}
 0 = \lim_{n\to\infty} \int_0^\vartheta \alpha_n(t)\|B^*z_n(t)\|_U^2\,dt = \int_0^\vartheta \alpha_\infty(t)\|B^*z_\infty(t)\|_U^2\,dt,
\end{equation}
so either $\alpha_\infty(t)=0$ or $B^*z_\infty(t)=0$ for almost every $t\in[0,\vartheta]$.
This proves \eqref{intvanishes} and hence $z_\infty(\cdot)$ solves, as claimed, the undamped equation \eqref{sysundamped}.

Finally, since $z_{\infty}$ solves \eqref{sysundamped}, we have $z_\infty(s)=e^{sA}z_\infty$ and thus \eqref{intvanishesnorm} implies \eqref{eq:lemvanish}, as required.
%
%
\end{proof}


In the example below we go back to the
internally damped Schr\"odinger equation considered in Example~\ref{exp_Sch_1D},
where we were not able to conclude whether such equation is strongly stable, uniformly with respect to all $T$-$\mu$ signals ($T>\mu>0$ given). We prove here, in the general $N$-dimensional case,
that weak stability holds true.

\begin{example}\label{weak_Sch} (Schr\"odinger equation)
Let $\Omega$ be a bounded domain of $\RR^N$, $N \geq 1$,
and consider the internally damped Schr\"odinger equation
\begin{align}
 	i y_t(t,x) +\laplace y(t,x) + i \alpha(t)d(x)^2 y(t,x)& = 0, 	\quad&(t,x)\in(0,\infty)\times \Omega,\label{schroedingerPDE1}\\
 	y(t,x)&=0	,					\quad& t\in (0,\infty)\times \partial \Omega,\label{schroedingerPDE2}\\
 	y(0,x) &= y_0(x),					\quad& t\in \Omega,	\label{schroedingerPDE3}
\end{align}
where $d(\cdot)$ belongs to $L^\infty(\Omega)$ and $\alpha(\cdot)$ is a $T$-$\mu$ PE-signal. Assume that
there exist $d_0 > 0$ and an open nonempty $\omega\subset\Omega$ such that
\begin{equation}\label{bcond}
|d(x)| \geq  d_0~\mbox{for a.\,e. $x$ in}~\omega.
\end{equation}

As in Example~\ref{exp_Sch_1D}, system \eqref{schroedingerPDE1}--\eqref{schroedingerPDE3}
can be written
in the form \eqref{sys} with $H=U=L^2(\Omega)$, $Az=i\laplace z$, $D(A)=\H^2(\Omega)\cap\H^1_0(\Omega)$, and $B\:z\mapsto d z$.
Since $A$ is skew-adjoint, it generates a contraction semigroup.

As to apply Theorem~\ref{thm:weak}, it remains to show that the pair $(A,B)$ has the $T$-$\mu$ PE unique continuation property stated in Hypothesis~\ref{hyp:thm1}.
To this end, fix some $z_0\in L^2(\Omega)$, some $T$-$\mu$ PE-signal $\alpha(\cdot)$, and choose any $\vartheta > T-\mu$. Observe that, since
\begin{equation*}
\int_0^\vartheta \alpha(t) \|B^* e^{tA}z_0\|^2_U\,dt = \int_0^\vartheta \alpha(t) \|d e^{tA}z_0\|^2_H\,dt,
\end{equation*}
then either $\alpha(t)=0$ or $d e^{tA}z_0=0$ for almost every $t\in[0,\vartheta]$. But \eqref{PEcondition} implies that $\alpha(\cdot)>0$ on a set $\Xi \subset (0,\vartheta)$ with $\text{meas}(\Xi)\geq\vartheta-T+\mu>0$ and \eqref{bcond} yields $d(\cdot)\ne 0$ a.\,e. on the open set $\omega\subset\Omega$. Hence,
\begin{equation*}
(t,x)\mapsto (e^{tA}z_0)(x) \equiv 0\quad\text{on}~\Xi \times \omega.
\end{equation*}

Let us now adapt the unique continuation argument proposed in \cite{ReiflerVogt1994} in order to prove that $(t,x)\mapsto (e^{t A}z_0)(x)=0$ on the open set $(0,\vartheta) \times \omega$.
Write the spectrum of $A$ as $(\lambda_k)_{k\in\NN}$ (with eigenvalues repeated according to their multiplicities). Then the sequence $(i\lambda_k)_{k\in\NN}$ is contained in $\RR$ and
is bounded from below.  Denote by $(\phi_k)_{k\in\NN}$ an orthonormal basis of $H$ such that $A\phi_k=\lambda_k \phi_k$.
Fix any $\varphi\in L^2(\omega)$ and consider the function
$F\:t\mapsto\sum_{k\in\NN}e^{\lambda_k t} \langle \phi_k,z_0\rangle \langle \phi_k,\varphi\rangle$.
Notice that $F(t)= \langle e^{t A}z_0,\varphi\rangle$ and that $F$ can be extended from $\RR$ to $\CC^-=\{w\in \CC\mid \mathrm{Im}(w)\leq 0\}$, thanks to the lower boundedness (in $\RR$) of $(i\lambda_k)_{k\in\NN}$.
 Moreover, $F$ is  complex analytic in the interior of $\CC^-$ and continuous up to its boundary.
Since $F$ is zero on a subset of the boundary of $\CC^-$  of positive (one-dimensional) measure, then it follows from Privalov's uniqueness theorem (see \cite[Vol.~II, Theorem 1.9, p. 203]{Zygmund})
that $F$ vanishes identically. By the arbitrariness of $\varphi  \in L^2(\omega)$ it follows, as required, that $e^{t A}z_0$ vanishes on $\omega$ for $t\in (0,\theta)$.

Applying Holmgren's uniqueness theorem (see \cite[Theorem 8.6.8]{Hormander}
and also \cite{Zuazua-remarks-2003}), we
deduce that
$z_0$ vanishes on $\Omega$, proving  Hypothesis~\ref{hyp:thm1}. \hfill$\diamond$
\end{example}

\section{Strong stability}\label{s:strong}
Condition \eqref{PEcondition} means that the feedback control $Bu=-BB^*z$ is, to some
extent, active on \emph{every} interval of the length $T$. From an application point
of view it is also interesting to study the case when there are intervals of arbitrary length
where no feedback control is active, in the spirit of the results in, e.\,g., \cite{HarauxMartinezVancostenoble2005,Hatvani1996,PucciSerrin1996,Smith1961}.
 A natural question is then to ask 
which conditions
imposed on $A$, $B$, and on the distribution and length of these intervals suffice to
ensure stability.

Below we give an abstract result ensuring the strong asymptotic stability of the closed-loop system \eqref{sys} using observability estimates for the open-loop system.
Stressing the importance, in order to apply such result, of having explicit estimates for control costs (i.e., the constants $c$ appearing in inequalities of the type \eqref{ineqneeded}), we then show on several examples how this can lead to stabilizing conditions.

\begin{definition}
We say that $\alpha(\cdot) \in L^\infty([0,T],[0,1])$ is of class $\K(A,B,T,c)$ if
\begin{equation}\label{classK}
 \int_0^T \alpha(t)\|B^* e^{sA}z_0\|_U^2\,dt \geq c\|z_0\|_H^2,\quad \text{for all}~z_0 \in H.
\end{equation}
\end{definition}

With this definition, we can state the following abstract result.

\begin{theorem}\label{thm:StrongStabAbstract}
Suppose that $(a_n,b_n)$, $n\in\NN$, is a sequence of disjoint intervals in $[0,\infty)$,
that $c_n$, $n\in\NN$, is a sequence of positive real numbers and that
$\alpha(\cdot) \in L^\infty([0,\infty),[0,1])$ is such that its restriction $\alpha(a_n+\cdot)|_{[0,b_n-a_n]}$ to
the interval $(a_n,b_n)$ is of class $\K(A,B,b_n-a_n,c_n)$ for all $n \in \NN$. Moreover, assume that
$\sup_{n\in\NN}(b_n-a_n) <\infty$ and $\sum_{n=1}^\infty c_n = \infty$.

Then the mild solution of \eqref{sys} satisfies $\|z(t)\|_H \to 0$ as $t \to \infty$.
\end{theorem}
\begin{proof}
First of all notice that \eqref{classK} implies that $c\leq T \|B^*\|^2$. Hence, the sum of the $c_n$ corresponding to intervals $(a_n,b_n)$ contained in a given bounded interval $[\tau_0,\tau_1]$ is finite and can be approximated arbitrarily well by the sum of finitely many of such $c_n$. Therefore, we can  extract a locally finite subsequence of intervals, still denoted by $(a_n,b_n)$, $n\in\NN$, such that
$\sum_{n=1}^\infty c_n=\infty$ and, up to a reordering,
$b_n\leq a_{n+1}$ for all $n\in\NN$.

Using the energy inequality \eqref{Vdecay} we get $V(z(a_{n+1})) \leq V(z(b_n))$ while Lemma~\ref{lem:lem0}, with $a=a_n$ and
$b=b_n$, implies that
\begin{equation*}
V(z(b_n))-V(z(a_n)) \leq-\frac{1}{2(1+(b_n-a_n)^2
 \|B\|^4)}\int_0^{b_n-a_n}\alpha(a_n+t)\|B^* e^{tA}z(a_n)\|_U^2\,dt.
\end{equation*}
Thus, since $\alpha(a_n+\cdot)|_{[0,b_n-a_n]}$ is of class $\K(A,B,b_n-a_n,c_n)$,
we have
\begin{equation}\label{eq:Vestn}
V(z(a_{n+1}))-V(z(a_n)) \leq -\frac{c_n}{1+(b_n-a_n)^2 \|B\|^4}V(z(a_n)).
\end{equation}

Using the estimate \eqref{eq:Vestn} recursively, we obtain
\begin{equation*}
V(z(a_{n+1})) \leq \prod_{j=1}^{n} \left(1-\frac{c_j}{1+(b_j-a_j)^2
 \|B\|^4}\right) V(z_0).
\end{equation*}
Since
\begin{align*}
\log \prod_{j=1}^{\infty} \left(1-\frac{c_j}{1+(b_j-a_j)^2  \|B\|^4}\right)
&=\sum_{j=1}^{\infty} \log\left(1-\frac{c_j}{1+(b_j-a_j)^2  \|B\|^4}\right)\\
&\leq -\sum_{j=1}^{\infty}\frac{c_j}{1+(b_j-a_j)^2  \|B\|^4}\\
&\leq -\frac{1}{1+ \|B\|^4\sup_{j=1}^\infty(b_j-a_j)^2 }\sum_{j=1}^{\infty}c_j=-\infty,
\end{align*}
then $V(z(a_{n+1}))$ tends to zero as $n$ goes to infinity.
\end{proof}

As a direct application of this abstract result, we consider again the Schr\"odinger equation
in one space dimension.

\begin{example}\label{EXAMP5.1}
(Schr\"odinger equation)
Consider
\begin{align}
 	i y_t(t,x) + y_{xx}(t,x) + i \alpha(t)d(x)^2y(t,x)& = 0,
 	\quad&(t,x)\in(0,\infty)\times(0,1),\label{schroedinger3PDE_1D1}\\
 	y(t,0)=y(t,1)&=0	,					\quad& t\in (0,\infty),\label{schroedinger3PDE_1D2}\\
 	y(0,x) &= y_0(x),					\quad& t\in(0,1),	\label{schroedinger3PDE_1D3}
\end{align}
with $d(\cdot) \in L^\infty(0,1)$ and $\alpha(\cdot) \in L^\infty([0,\infty),[0,1])$. Assume that
$d=\chi_\omega$ with $\omega$ a nonempty subinterval of $(0,1)$ and assume that
$(a_n,b_n)$, $n\in\NN$, is a sequence of disjoint intervals in $[0,\infty)$ such that
$\sup_{n\in\NN}(b_n-a_n)<\infty$ and $\alpha(\cdot)|_{(a_n,b_n)}
\equiv 1$.

As in Example~\ref{exp_Sch_1D}, we write system
\eqref{schroedinger3PDE_1D1}--\eqref{schroedinger3PDE_1D3} in the form
\eqref{sys} with $H=U=L^2(0,1)$, the skew-adjoint operator $A$ given by
$Az=iz_{xx}$ acting on $D(A)=\H^2(0,1)\cap\H^1_0(0,1)$, and the
multiplication operator $B\:z\mapsto \chi_\omega z$, so that for $y_0 \in H$, the mild
solution $z(\cdot)$ of \eqref{sys} with this choice of $A,B$ corresponds to the
weak solution $y(\cdot)$ of \eqref{schroedinger3PDE_1D1}--\eqref{schroedinger3PDE_1D3}.

It is well known that for any interval $(a_n,b_n)$, $n\in\NN$, there exists a
positive constant $c_n$ such that
\begin{equation}\label{eq:obsineq_n}
\int_{a_n}^{b_n} \int_{\omega} |e^{tA}z(x)|^2\,dx\,dt \geq c_n\|z\|_H^2,~\quad
z\in H,
\end{equation}
that is, $\alpha(a_n+\cdot)|_{[0,b_n-a_n]}$ is of class
$\K(A,B,b_n-a_n,c_n)$
(see, for instance, \cite[Remark~6.5.4]{TucsnakWeiss2009}). Moreover, rewriting
\eqref{eq:obsineq_n} as
\begin{equation*}
\int_{a_n}^{b_n} \int_{\omega} \left|\sum_{k \in \NN} \la\phi_k,z\ra_{L^2(0,1)}
\phi_k(x) e^{i n^2 \pi^2 t} \right|^2\,dx\,dt \geq c_n\|z\|_H^2,
\end{equation*}
with $\phi_k(x)=\sqrt{2}\sin(n \pi x)$ we get from
\cite[Corollary~3.2]{TenenbaumTucsnak2007} that $c_n$ can be taken satisfying
\begin{equation*}
c_n \geq C e^{-\frac{2}{b_n-a_n}}
\end{equation*}
for
some positive constant $C$ independent of $n$.

Hence,
Theorem~\ref{thm:StrongStabAbstract} guarantees that the mild solution
of \eqref{sys} converges strongly to the origin in $H$ if
$$
\sum_{n=1}^\infty e^{-\frac{2}{b_n-a_n}}= \infty.
$$
\hfill$\diamond$
\end{example}

\begin{remark}
The results in the above example and, more generally, the methodology employed in this section, 
can be adapted to the case of some unbounded control operators and thus to boundary stabilization problems.
As an example, consider the system
\begin{align}
 	i y_t(t,x) + y_{xx}(t,x) & = 0,
 	\quad&(t,x)\in(0,\infty)\times(0,1),\label{schroedinger3PDE_1D1_BD}\\
 y_x(t,0)&=-i\alpha(t){y(t,0)}	,					\quad& t\in (0,\infty),\label{schroedinger3PDE_1D2_BD_BIS}\\
 	y(t,1)&=0	,					\quad& t\in (0,\infty),\label{schroedinger3PDE_1D2_BD}\\
 	y(0,x) &= y_0(x),					\quad& x\in(0,1).	\label{schroedinger3PDE_1D3_BD}
\end{align}
with a piecewise constant $\alpha\: [0,\infty) \to [0,1]$ satisfying conditions as in
Example~\ref{EXAMP5.1} for some sequences $(a_n)_{n\in\NN}$, $(b_n)_{n\in\NN}$.
Clearly, Theorem~\ref{thm:StrongStabAbstract} does not apply in this case. 
However, we can retrieve similar 
conditions on the intervals $(a_n,b_n)$
in order to have the strong stability property as in Example~\ref{EXAMP5.1}. 
Indeed, it suffices to check  the exact observability for the undamped dynamics 
and to show that an energy estimate such in Lemma~\ref{lem:lem0} holds for the constant damping case.
The operator $A\: D(A)\to L^2(0,1)$ corresponding to the undamped case (i.e., $\alpha=0$ 
in \eqref{schroedinger3PDE_1D2_BD_BIS}) is
$$
D(A)=\{\varphi\in H^2(0,1) \mid \varphi_x(0)=0,\ \varphi(1)=0\},
$$
$$
A\varphi=i\varphi_{xx} \qquad(\varphi\in D(A)),
$$
whereas the control operator is given by $B=\delta_0$, where $\delta_0$ is the Dirac mass concentrated at the origin.

Using the results in \cite{TenenbaumTucsnak2007}, it is not difficult to check that for $\alpha=0$ there exist $C_1,\ C_2>0$ such that 
$$
C_1 e^{\frac{C_2}{T}}\int_0^T |y(t,0)|^2\, {\rm d}t\geqslant \|y_0\|_{L^2(0,1)}^2 \quad(T>0,\ y_0\in D(A)).
$$
The last formula is, according to the above definitions of $A$ and $B$, equivalent to the inequality 
\begin{equation}\label{classK_BIS}
 C_1 e^{\frac{C_2}{T}} \int_0^T |B^* e^{sA}y_0|\,dt \geq c\|y_0\|_{L^2(0,1)} \qquad (y_0 \in D(A)),
\end{equation}
so that we have indeed the exact observability in any time $T>0$ for the undamped dynamics. 

To check an energy estimate similar to the one in Lemma~\ref{lem:lem0}, one can first prove \eqref{Vdecay} for $y_0$ 
in the domain of the generator (which is done via integration by parts). One can then check (using, for instance, 
a transfer function like in Guo and Shao \cite{GuoShao}) that the system $(A,B,B^*)$ is well-posed 
in the sense of Salamon and Weiss (see \cite{Weiss10}). \hfill$\diamond$
\end{remark}

Sufficient conditions for strong stability as those obtained in 
Theorem~\ref{thm:StrongStabAbstract} can be specified 
more precisely in the case of 
integral ``excitations''.

\begin{theorem}\label{thm:c(T)}
{
Suppose that there exist constants $\rho,T_0>0$ and a positive, continuous
function $c\: (0,\infty) \to \RR$ such that for all $T\in(0,T_0]$, if for
some $\tilde\alpha \in L^\infty([0,T],[0,1])$
\begin{equation*}
 \int_0^T \tilde{\alpha}(t)\,dt \geq \rho T
\end{equation*}
then 
$\tilde{\alpha}(\cdot)$ is of class $\K(A,B,T,c(T))$.
Let $(a_n,b_n)$, $n\in\NN$, be  a sequence of disjoint intervals in $[0,\infty)$
and $\alpha \in L^\infty([0,\infty),[0,1])$. Assume that
$\int_{a_n}^{b_n} \alpha(t)\,dt \geq
\rho (b_n-a_n)$ and $\sum_{n=1}^\infty c(b_n-a_n)=\infty$.
Then the mild solution of \eqref{sys} satisfies $\|z(t)\|_H \to 0$ as $t \to
\infty$.}
\end{theorem}
\begin{proof}\mbox{}
In the case where $\sup_{n\in\NN}(b_n-a_n) \leq T_0$ the conclusion follows directly from
Theorem~\ref{thm:StrongStabAbstract}.

Now assume  that for infinitely many $n\in\NN$, $b_n-a_n>T_0$.
{
Let $n$ be
such that
$b_n-a_n>T_0$
and split $I_n=(a_n,b_n)$ into finitely many
pairwise disjoint subintervals $I_n^1,\dots,I_n^r$ of common length $l_n\in[T_0/2,T_0]$.
Since $\sum_{j=1}^r \int_{I_n^j}\alpha(t)\,dt\geq \rho (b_n-a_n)=
r\rho  \,l_n$, then there exists $j\in\{1,\dots,r\}$ such that $\int_{I_n^j}\alpha(t)\,dt\geq
\rho l_n=\rho |I_n^j|$.
Denote $I_n^j$ by $(a_n',b_n')$.
If $n$ is such that $b_n-a_n\leq T_0$, set  $a_n'=a_n$ and $b_n'=b_n$.
}

Again applying  Theorem~\ref{thm:StrongStabAbstract} to the sequence of intervals $(a_n',b_n')$, $n\in\NN$, we
can conclude by showing that
$\sum_{n=1}^{\infty}c(b_n'-a_n')=\infty$, since $\sup_{n\in\NN} (b_n'-a_n')<\infty$.
The unboundedness of $\sum_{n=1}^{\infty}c(b_n'-a_n')$ follows from the remark that, for infinitely many $n\in\NN$,
$c(b_n'-a_n')\geq \min_{T\in[T_0/2,T_0]} c(T)>0$.
\end{proof}

 \begin{example} (Wave equation) 
 %
Let $t\mapsto z(t)=(v(t,\cdot),v_t(t,\cdot))$ be a solution of the wave equation
\eqref{wellengleichung_2}--\eqref{RB_2}
where, as in Example~\ref{ex_wave_2}, $\Omega$ is a bounded domain in $\RR^N$,
$N\geq 1$, $d\in L^\infty(\Omega)$
satisfies
\begin{equation}\label{example-stability-3}
|d(x)|\geq  d_0>0\quad \mbox{for almost all}~x\in \Omega.
\end{equation}
Consider $T,\rho>0$ and some  $\alpha(\cdot)\in L^\infty([0,T],[0,1])$
satisfying
\begin{equation}\label{example-alpha-cond-3}
\int_0^T \alpha(t)\,dt \geq T \rho.
\end{equation}

Using the same notation as in Example~\ref{ex_wave_2} and, in particular,
fixing an initial condition and defining the set $A_n^\epsilon$ as in \eqref{eq:AepsDef}, we have, according to \eqref{meas--},
\begin{equation*}
\text{meas}(A_n^\epsilon) \geq
T\left(1-\frac{2\epsilon}\pi\right)-2\frac\epsilon{\min_{n\in\NN}\lambda_n}
\end{equation*}
for any $\epsilon \in (0,1)$. Without loss of generality, we can assume that $\min_{n\in\NN}\lambda_n=\lambda_1$.
For $T$ small enough, choosing
$\bar{\epsilon}=\frac{\rho\lambda_1}{6} T$, we get $\text{meas}(A_n^\epsilon) \geq T\left(1- \frac {\rho}{2}\right)$, leading to
\begin{equation*}
\int_{A_n^{\bar\epsilon}} \alpha(t)\,dt \geq \frac {T\rho}{2},
\end{equation*}
because of \eqref{example-alpha-cond-3}.
The definition of $A_n^{\bar \epsilon}$ yields the observability estimate
\begin{align*}
\int_0^T \alpha(t)( -a_n\sin({ \sqrt{\lambda_n} t}) + b_n\cos({\sqrt{\lambda_n}
t}))^2\,dt&\geq {\bar\epsilon}^2\, (a_n^2+b_n^2)\int_{A_n^{\bar\epsilon}}
\alpha(t)\,dt\\ &\geq \frac{\rho^3 \lambda_1^2}{72} T^3 (a_n^2+b_n^2).
\end{align*}
Reasoning as in Example~\ref{ex_wave_2},
we obtain that the function $c(T)$ appearing in the statement of
Theorem~\ref{thm:c(T)} for the system
\eqref{wellengleichung_2}--\eqref{RB_2} with uniform damping
\eqref{example-stability-3} can be chosen of order $T^3$ for $T$ small.

In particular, Theorem~\ref{thm:c(T)} states 
that
a sufficient condition for the
strong asymptotic stability of the solutions
of \eqref{wellengleichung_2}--\eqref{RB_2} with uniform damping
\eqref{example-stability-3} is that
$\alpha(\cdot)\in L^\infty([0,\infty),[0,1])$
satisfies 
\begin{equation*}
\int_{a_n}^{b_n} \alpha(t)\,dt \geq \rho (b_n-a_n),\quad n\in\NN,
\end{equation*}
for some
positive constant $\rho$ and some sequence $(a_n,b_n)$, $n\in\NN$, of disjoint intervals in
$[0,\infty)$ such that
\begin{equation*}
\sum_{n=1}^\infty (b_n-a_n)^3=\infty.
\end{equation*}

{
For the particular case of the wave equation, the sufficient
condition obtained here weakens the one considered in
\cite{HarauxMartinezVancostenoble2005} where $\alpha(\cdot)$ is bounded away
from 0 by a constant $m_n$ on each interval $(a_n,b_n)$ in order to guarantee
asymptotic stability (cf. also Remark~\ref{rem:haraux}).
}

\hfill$\diamond$
\end{example}

\begin{example}\mbox{}
 ({
 Finite-dimensional linear systems})
Let us characterize the function $c(\cdot)$ appearing in the statement of Theorem~\ref{thm:c(T)} in the case of finite-dimensional systems, that is, when  $\dim (H)<\infty$.
We prove below that $T$ behaves polynomially and that its degree for $T$ small
can be taken
equal to the sharp estimate computed by Seidman in the case $\alpha\equiv 1$ (see \cite{Seidman1988}).

In the finite-dimensional case, the assumption that $A$ generates a strongly continuous contraction semigroup
is standardly weakened 
into the requirement that  $A$ is neutrally stable, that is, its eigenvalues are
of non-positive real part and all Jordan blocks corresponding  to pure imaginary eigenvalues are trivial.

Clearly, a necessary condition for ensuring the convergence to the origin of all trajectories of the system $\dot x=Ax+\alpha B u$ for some  $\alpha=\alpha(t)\in[0,1]$ is that the pair $(A,B)$ is stabilizable. We will make this assumption from now on.

 Up to a linear change of
variables, $A$ and $B$ can be written as
$$
A=\begin{pmatrix}
A_1&A_2\\
0&A_3
\end{pmatrix}, \ \ \
B=\begin{pmatrix} B_1\\B_3
\end{pmatrix},
$$
where $A_1$ is Hurwitz and all the eigenvalues of $A_3$
are purely imaginary.
 From the neutral stability assumption and up to a further linear change of
coordinates, we may assume that $A_3$ is skew-symmetric.
 From the stabilizability  assumption on $(A,B)$, moreover, we deduce that $(A_3,B_3)$ is
controllable.

Setting $x=(x_1,x_3)$ according to the above decomposition,
the system $\dot x=Ax+\alpha B u$ can be written as
\begin{align}
\dot x_1&=A_1x_1+A_2x_3+\alpha(t)B_1u,\label{h1}\\
\dot x_3&=A_3x_3+\alpha(t)B_3u.\label{h2}
\end{align}

 Assume that, for a given $\alpha(\cdot)$, all solutions of \eqref{h2} with $u=-B_3^\top x_3$
 converge to the origin.
Then all trajectories of system \eqref{h1}-\eqref{h2}, with the  choice of feedback $u=-B_3^\top x_3$,
converge to the origin, since \eqref{h1} becomes
an autonomous linear Hurwitz
system subject to a perturbation
whose norm converges to zero as time goes to infinity.

The previous discussion allows us to focus on the special case where
$A$ is skew-symmetric and $(A,B)$ is controllable.

Denote by
$K_{(A,B)}$ the minimal non-negative integer such that
\begin{equation}\label{K-alman}
\mathrm{rank}[B,AB,\dots,A^{K_{(A,B)}} B]=N, 
\end{equation}
 where $N$ is the dimension of $H$.
We have the following result.

\begin{proposition}\label{prop:finite}
Let $A$ be skew-symmetric and $(A,B)$ controllable.
Then for every $\rho>0$ there exists $\kappa>0$ such that, for every $T\in (0,1]$ and
every $\alpha\in L^\infty([0,T],[0,1])$, if $\int_0^T\alpha(t)dt\geq \rho T$ then
$\alpha$ is of class $\K(A,B,T,\kappa T^{2K_{(A,B)}+1})$.
 \end{proposition}

 \begin{proof}
Let $K=K_{(A,B)}$ and fix $\rho>0$. We should prove that, for some $\kappa>0$, given any $z_0\in \RR^n$ and any $\alpha\in L^\infty([0,T],[0,1])$ such that $T\in (0,1]$ and $\int_0^T\alpha(t)dt\geq \rho T$, we have
$$\int_0^T \alpha(t)\|B^\top e^{t A} z_0\|^2dt\geq \kappa T^{2K+1} \|z_0\|^2.$$

Denote by $b_1,\dots,b_r$ the columns of $B$ and assume, by contradiction, that there exist $(T_n)_{n\in \NN}\subset (0,1]$, $(z_0^n)_{n\in \NN}\subset \RR^{N}$ with $\|z_0^n\|= 1$, and $(\alpha^n)_{n\in\NN}\subset L^\infty([0,1],[0,1])$ with $\int_0^{T_n}\alpha^n(t)dt\geq \rho T_n$ such that $\lim_{n\to\infty}\kappa_n=0$ where
 $$\kappa_n=\frac{\int_0^{T_n}\alpha^n(t)\sum_{i=1}^r(b_i^\top e^{t A} z_0^n)^2 dt}{T_n^{2K+1}},\quad n\in\NN.$$

Let $\beta^n(t)=\alpha^n(T_n t)$ for $n\in\NN$ and $t\in[0,1]$.
Then $\int_0^1 \beta^n(t)dt\geq \rho$  and
$$\kappa_n=\frac{\int_0^{1}\beta^n(t)\sum_{i=1}^r(b_i^\top e^{t T_n A} z_0^n)^2 dt}{T_n^{2K}},\quad n\in\NN.$$

By compactness, up to extracting a subsequence, $T_n\to T_\infty$ in $[0,1]$, $z_0^n\to z_0^\infty$ in 
$\RR^{N}$ and $\beta^n\weakstarto \beta^\infty$ in $L^\infty([0,1],[0,1])$.
In particular, $\|z_0^\infty\|=1$, $\int_0^1 \beta^\infty(t)dt\geq \rho$ and
$$
\lim_{n\to \infty}\kappa_n T_n^{2K}=\int_0^{1}\beta^\infty(t) \sum_{i=1}^r\left(b_i^\top e^{t T_\infty A} z_0^\infty\right)^2 dt=0.
$$

Assume first that  $T_\infty>0$. Then the analytic function $t\mapsto \sum_{i=1}^r\left(b_i^\top e^{t T_\infty A} z_0^\infty\right)^2$ annihilates on  the support of $\beta^\infty$, which has positive measure, and is thus identically equal to zero, contradicting the controllability of the pair $(A,B)$.

Let then $T_\infty=0$.
Rewrite $\kappa_n$ as
$$\kappa_n=\int_0^1 \beta^n(t)\sum_{i=1}^r\left( c_0^{i,n}+t c_1^{i,n}+\cdots+t^K c_K^{i,n}+r^{i,n}(t)\right)^2 dt,$$
where
$$c_j^{i,n}=\frac{b_i^\top A^j z_0^n}{j!T_n^{K-j}},\qquad \|r^{i,n}\|_{L^\infty(0,1)}\leq M T_n,$$
 for some $M>0$ only depending on $A$, $B$, and $K$.
Define the vector
$$C^n=(c_0^{1,n},\dots,c_K^{1,n},c_0^{2,n},\dots,c_K^{2,n},\dots,c_0^{r,n},\dots,c_K^{r,n})$$
belonging to $\RR^{r(K+1)}$.
Since $\|z_0^n\|=1$, $T_n\leq 1$, and because of \eqref{K-alman},  there exists $\nu>0$ only depending on $A$ and $B$ such that
$\|C^n\|\geq \nu$.
Thus,
$$\kappa_n\geq \nu^2 \frac{\int_0^1 \beta^n(t)\sum_{i=1}^r\left( c_0^{i,n}+t c_1^{i,n}+\cdots+t^K c_K^{i,n}+r^{i,n}(t)\right)^2 dt}{\|C^n\|^2}.$$
Up to extracting a subsequence, $C^n/\|C^n\|$ converges in the unit sphere of $\RR^{r(K+1)}$. Denote its limit by $(\gamma_0^{1},\dots,\gamma_K^{r})$. Then
$$t\mapsto \sum_{i=1}^r\left( c_0^{i,n}+t c_1^{i,n}+\cdots+t^K c_K^{i,n}+r^{i,n}(t)\right)^2$$ converges uniformly on $[0,1]$ to
$t\mapsto\sum_{i=1}^r\left( \gamma_0^{i}+t \gamma_1^{i}+\cdots+t^K \gamma_K^{i}\right)^2$. We can conclude that
$$\int_0^1 \beta^\infty(t)\sum_{i=1}^r\left( \gamma_0^{i}+t \gamma_1^{i}+\cdots+t^K \gamma_K^{i}\right)^2dt=0,$$
leading to a contradiction, since $\beta^\infty$ is nonzero on a subset of $[0,1]$ of positive measure and $(\gamma_0^{1},\dots,\gamma_K^{r})$ is a nonzero vector.
\end{proof}

Proposition~\ref{prop:finite} and Theorem~\ref{thm:c(T)} imply the following.

\begin{corollary}
Let $A$ be skew-symmetric and $(A,B)$ be controllable.
Then for every $\rho>0$,
every $\alpha\in L^\infty([0,\infty),[0,1])$ such that there exist a sequence
$(a_n,b_n)$, $n\in\NN$,
of disjoint intervals in $[0,\infty)$
with 
$\int_{a_n}^{b_n} \alpha(t)\,dt \geq
\rho (b_n-a_n)$ and $\sum_{n=1}^\infty (b_n-a_n)^{2K_{(A,B)}+1}=\infty$, and every solution $z(\cdot)$ of \eqref{sys} corresponding to $\alpha$,
we have $\|z(t)\|_{\RR^N} \to 0$ as $t \to \infty$.
\end{corollary}

\hfill$\diamond$
\end{example}

{\bf Acknowledgments.}
 This work was supported by the ANR grant ArHyCo, Program ARPEGE, contract number
ANR-2008 SEGI 004 01-30011459. The research presented in this article was mostly carried out
while F.~M.~Hante and M.~Sigalotti were
with Institut
\'Elie Cartan (IECN)
 and CORIDA, INRIA Nancy--Grand Est.

\bibliographystyle{siam}
\bibliography{PE}

\begin{thebibliography}{10}

\bibitem{AeyelsPeuteman1998}
{\sc D.~Aeyels and J.~Peuteman}, {\em A new asymptotic stability criterion for
  nonlinear time-variant differential equations}, IEEE Trans. Automat. Control,
  43 (1998), pp.~968--971.

\bibitem{ABJKKMPR}
{\sc B.~Anderson, R.~Bitmead, C.~Johnson, P.~Kokotovic, R.~Kosut, I.~Mareels,
  L.~Praly, and B.~Riedle}, {\em Stability of adaptive systems: Passivity and
  averaging analysis}, MIT Press, 1986.

\bibitem{Ball1977}
{\sc J.~M. Ball}, {\em Strongly continuous semigroups, weak solutions, and the
  variation of constants formula}, Proc. Amer. Math. Soc., 63 (1977),
  pp.~370--373.

\bibitem{BallMarsdenSlemrod1982}
{\sc J.~M. Ball, J.~E. Marsden, and M.~Slemrod}, {\em Controllability for
  distributed bilinear systems}, SIAM J. Control Optim., 20 (1982),
  pp.~575--597.

\bibitem{BallSlemrod1979}
{\sc J.~M. Ball and M.~Slemrod}, {\em Feedback stabilization of distributed
  semilinear control systems}, Appl. Math. Optim., 5 (1979), pp.~169--179.

\bibitem{GuzmanTucsnak2003}
{\sc R.~Benavides~Guzm{\'a}n and M.~Tucsnak}, {\em Energy decay estimates for
  the damped plate equation with a local degenerated dissipation}, Systems
  Control Lett., 48 (2003), pp.~191--197.
\newblock Optimization and control of distributed systems.

\bibitem{ChailletChitourLoriaSigalotti2008}
{\sc A.~Chaillet, Y.~Chitour, A.~Lor{\'{\i}}a, and M.~Sigalotti}, {\em Uniform
  stabilization for linear systems with persistency of excitation: the
  neutrally stable and the double integrator cases}, Math. Control Signals
  Systems, 20 (2008), pp.~135--156.

\bibitem{ChitourSigalotti2010}
{\sc Y.~Chitour and M.~Sigalotti}, {\em On the stabilization of persistently
  excited linear systems}, SIAM J. Control Optim., 48 (2010), pp.~4032--4055.

\bibitem{Fattorini2005}
{\sc H.~O. Fattorini}, {\em Infinite dimensional linear control systems},
  vol.~201 of North-Holland Mathematics Studies, Elsevier Science B.V.,
  Amsterdam, 2005.
\newblock The time optimal and norm optimal problems.

\bibitem{FragnelliMugnai2008}
{\sc G.~Fragnelli and D.~Mugnai}, {\em Stability of solutions for some classes
  of nonlinear damped wave equations}, SIAM J. Control Optim., 47 (2008),
  pp.~2520--2539.

\bibitem{GuoShao}
{\sc B.-Z. Guo and Z.-C. Shao}, {\em Regularity of a {S}chr\"odinger equation
  with {D}irichlet control and colocated observation}, Systems Control Lett.,
  54 (2005), pp.~1135--1142.

\bibitem{HanteSigalotti2011}
{\sc F.~M. Hante and M.~Sigalotti}, {\em Converse {L}yapunov theorems for
  switched systems in {B}anach and {H}ilbert spaces}, SIAM J. Control Optim.,
  49 (2011), pp.~752--770.

\bibitem{Haraux1989}
{\sc A.~Haraux}, {\em Une remarque sur la stabilisation de certains syst\`emes
  du deuxi\`eme ordre en temps}, Portugal. Math., 46 (1989), pp.~245--258.

\bibitem{HarauxMartinezVancostenoble2005}
{\sc A.~Haraux, P.~Martinez, and J.~Vancostenoble}, {\em Asymptotic stability
  for intermittently controlled second-order evolution equations}, SIAM J.
  Control Optim., 43 (2005), pp.~2089--2108.

\bibitem{Hatvani1996}
{\sc L.~Hatvani}, {\em Integral conditions on the asymptotic stability for the
  damped linear oscillator with small damping}, Proc. Amer. Math. Soc., 124
  (1996), pp.~415--422.

\bibitem{HatvaniKrisztinTotik1995}
{\sc L.~Hatvani, T.~Krisztin, and V.~Totik}, {\em A necessary and sufficient
  condition for the asymptotic stability of the damped oscillator}, J.
  Differential Equations, 119 (1995), pp.~209--223.

\bibitem{Hormander}
{\sc L.~H{\"o}rmander}, {\em The analysis of linear partial differential
  operators. {I}}, Classics in Mathematics, Springer-Verlag, Berlin, 2003.
\newblock Distribution theory and Fourier analysis, Reprint of the second
  (1990) edition.

\bibitem{Martinez1999}
{\sc P.~Martinez}, {\em Decay of solutions of the wave equation with a local
  highly degenerate dissipation}, Asymptot. Anal., 19 (1999), pp.~1--17.

\bibitem{MartinezVancostenoble2002}
{\sc P.~Martinez and J.~Vancostenoble}, {\em Stabilization of the wave equation
  by on-off and positive-negative feedbacks}, ESAIM Control Optim. Calc. Var.,
  7 (2002), pp.~335--377.

\bibitem{MizelSeidman1997}
{\sc V.~J. Mizel and T.~I. Seidman}, {\em An abstract bang-bang principle and
  time-optimal boundary control of the heat equation}, SIAM J. Control Optim.,
  35 (1997), pp.~1204--1216.

\bibitem{MORNAR}
{\sc A.~Morgan and K.~Narendra}, {\em On the stability of nonautonomous
  differential equations $\dot x = (a+b(t))x$ with skew-symmetric matrix
  $b(t)$}, SIAM J. Control Optim., 15 (1977), pp.~163--176.

\bibitem{Nakao1996}
{\sc M.~Nakao}, {\em Decay of solutions of the wave equation with a local
  degenerate dissipation}, Israel J. Math., 95 (1996), pp.~25--42.

\bibitem{Pazy1983}
{\sc A.~{Pazy}}, {\em Semigroups Of Linear Operators And Applications To
  Partial Differential Equations}, Applied Mathematical Sciences Series,
  Springer-Verlag, New York, 1983.

\bibitem{Phung}
{\sc K.~D. Phung and G.~Wang}, {\em An observability estimate for parabolic
  equations from a measurable set in time and its applications}, preprint,
  (2011).

\bibitem{PucciSerrin1996}
{\sc P.~Pucci and J.~Serrin}, {\em Asymptotic stability for nonautonomous
  dissipative wave systems}, Comm. Pure Appl. Math., 49 (1996), pp.~177--216.

\bibitem{ReiflerVogt1994}
{\sc F.~Reifler and A.~Vogt}, {\em Unique continuation of some dispersive
  waves}, Comm. Partial Differential Equations, 19 (1994), pp.~1203--1215.

\bibitem{Seidman1986}
{\sc T.~I. Seidman}, {\em The coefficient map for certain exponential sums},
  Nederl. Akad. Wetensch. Indag. Math., 48 (1986), pp.~463--478.

\bibitem{Seidman1988}
\leavevmode\vrule height 2pt depth -1.6pt width 23pt, {\em How violent are fast
  controls?}, Math. Control Signals Systems, 1 (1988), pp.~89--95.

\bibitem{Slemrod1974}
{\sc M.~Slemrod}, {\em A note on complete controllability and stabilizability
  for linear control systems in {H}ilbert space}, SIAM J. Control, 12 (1974),
  pp.~500--508.

\bibitem{Smith1961}
{\sc R.~A. Smith}, {\em Asymptotic stability of
  {$x^{\prime\prime}+a(t)x^{\prime} +x=0$}}, Quart. J. Math. Oxford Ser. (2),
  12 (1961), pp.~123--126.

\bibitem{Tcheugoue1998}
{\sc L.~R. Tcheugou{\'e}~T{\'e}bou}, {\em On the decay estimates for the wave
  equation with a local degenerate or nondegenerate dissipation}, Portugal.
  Math., 55 (1998), pp.~293--306.

\bibitem{TenenbaumTucsnak2007}
{\sc G.~Tenenbaum and M.~Tucsnak}, {\em New blow-up rates for fast controls of
  {S}chr\"odinger and heat equations}, J. Differential Equations, 243 (2007),
  pp.~70--100.

\bibitem{TucsnakWeiss2009}
{\sc M.~Tucsnak and G.~Weiss}, {\em Observation and control for operator
  semigroups}, Birkh\"auser Advanced Texts: Basel Textbooks, Birkh\"auser
  Verlag, Basel, 2009.

\bibitem{Wang2008}
{\sc G.~Wang}, {\em {$L^\infty$}-null controllability for the heat equation and
  its consequences for the time optimal control problem}, SIAM J. Control
  Optim., 47 (2008), pp.~1701--1720.

\bibitem{Weiss10}
{\sc G.~Weiss}, {\em {T}ransfer functions of regular linear systems, {P}art
  {I}: {C}haracterizations of regularity}, Trans. Amer. Math. Society, 342
  (1994), pp.~827--854.

\bibitem{Zuazua-remarks-2003}
{\sc E.~Zuazua}, {\em Remarks on the controllability of the {S}chr\"odinger
  equation}, in Quantum control: mathematical and numerical challenges, vol.~33
  of CRM Proc. Lecture Notes, Amer. Math. Soc., Providence, RI, 2003,
  pp.~193--211.

\bibitem{Zygmund}
{\sc A.~Zygmund}, {\em Trigonometric series. {V}ol. {I}, {II}}, Cambridge
  Mathematical Library, Cambridge University Press, Cambridge, third~ed., 2002.

\end{thebibliography}

\end{document}